\documentclass[11pt, leqno]{amsart}
\usepackage{amsmath}
\usepackage{amsfonts}
\usepackage{amssymb}
\usepackage{dsfont}
\usepackage{graphicx}
\usepackage{mathrsfs}
\usepackage[margin=3.20cm]{geometry}
\usepackage{url}
\usepackage{hyperref}
\providecommand{\U}[1]{\protect\rule{.1in}{.1in}}
\newtheorem{theorem}{Theorem}[section]
\newtheorem*{acknowledgement*}{Acknowledgements}

\newtheorem{corollary}[theorem]{Corollary}

\newtheorem{lemma}[theorem]{Lemma}

\newtheorem{remark}[theorem]{Remark}

\newtheorem{theoremA}{Theorem} 
\newtheorem{corollaryA}[theoremA]{Corollary}

\def\Ric{{\mathrm{Ric}}}

\def\<{\left\langle}
\def\>{\right\rangle}
\newcommand{\abs}[1]{\lvert{#1}\rvert}
\newcommand{\scal}[2]{\langle{#1},{#2}\rangle}

\newcommand{\matrice}{\begin{pmatrix}}
\newcommand{\ok}{\end{pmatrix}}
\newcommand{\dmatrice}{\begin{vmatrix}}
\newcommand{\dok}{\end{vmatrix}}

\newcommand{\twosystem}[2]{\left\{\begin{aligned} &#1\\ &#2\end{aligned}\right.}
\newcommand{\threesystem}[3]{\left\{ \begin{aligned}&#1\\ &#2\\&#3\end{aligned}\right.}

\newcommand{\nero}{\smallskip$\bullet\quad$\rm}

\newcommand{\Cal }[1]{{\mathcal {#1}}}
\newcommand{\norm}[1]{\lVert{#1}\rVert}

\newcommand{\reals}{{\bf R}}
\newcommand{\sphere}[1]{{\bf S}^{#1}}

\newcommand{\real}[1]{{\bf R}^{#1}}

\newcommand{\tr}{{\rm tr}}

\newcommand{\bd}{\partial}

\newcommand{\derive}[2]{\dfrac{\bd #1}{\bd#2}}

\def\Ric{{\mathrm{Ric}}}

\def\<{\left\langle}
\def\>{\right\rangle}

\begin{document}
\title[Index and first Betti number of $f$-minimal hypersurfaces]
{Index and first Betti number of $f$-minimal hypersurfaces and self-shrinkers}


\author[Debora Impera]{Debora Impera}
\address[Debora Impera]{Dipartimento di Scienze Matematiche "Giuseppe Luigi Lagrange", Politecnico di Torino, Corso Duca degli Abruzzi, 24, Torino, Italy, I-10129}
\email{debora.impera@gmail.com}

\author[Michele Rimoldi]{Michele Rimoldi}
\address[Michele Rimoldi]{Dipartimento di Scienze Matematiche "Giuseppe Luigi Lagrange", Politecnico di Torino, Corso Duca degli Abruzzi, 24, Torino, Italy, I-10129}
\email{michele.rimoldi@gmail.com}

\author[Alessandro Savo]{Alessandro Savo}
\address[Alessandro Savo]{Dipartimento SBAI, Sezione di Matematica, Sapienza Universit\`a di Roma, Via Antonio Scarpa, 16, Roma, Italy, I-00161}
\email{alessandro.savo@sbai.uniroma1.it}

\begin{abstract}

We study the Morse index of self-shrinkers for the mean curvature flow and, more generally, of $f$-minimal hypersurfaces in a weighted Euclidean space endowed with a convex weight. 
When the hypersurface is compact, we show that the index is bounded from below by an affine function of its first Betti number. When the first Betti number is large, this improves index estimates known in literature. In the complete non-compact case, the lower bound is in terms of the dimension of the space of weighted square summable $f$-harmonic $1$-forms; in particular, in dimension $2$,  the procedure gives an index estimate in terms of the genus of the surface.
\end{abstract}

\maketitle

\tableofcontents

\section{Introduction}

\subsection{Main definitions} 

It is well-known that an immersed hypersurface $\Sigma^m$ of a given Riemannian manifold $(M^{m+1},g)$ is minimal (i.e. it has everywhere vanishing mean curvature) if and only if it is critical for the volume  functional 
$$
\Sigma\mapsto {\rm vol}(\Sigma)=\int_{\Sigma}\,d\mu
$$ 
where $d\mu$ is the Riemannian measure associated to the induced metric $g$ on $\Sigma$. More generally, given a smooth function $f$ on $M^{m+1}$, one can consider the 
so-called {\it $f$-volume}:
$$
{\rm vol}_f(\Sigma)=\int_{\Sigma}e^{-f}\,d\mu.
$$ 
It is then natural to study the immersions $x:\Sigma^m\to M^{m+1}$ which are critical for the $f$-volume. For the sake of exposition assume, for the moment, that $\Sigma$ is compact and orientable.  Given  $u\in C^{\infty}(\Sigma)$ and a unit normal vector field $N$ on $\Sigma$, consider the associated one-parameter deformation:
\begin{equation}\label{opv}
\Sigma_t\doteq \{\exp_x(tu(x)N(x)): x\in\Sigma\}
\end{equation}
for $t\in (-\epsilon,\epsilon)$ sufficiently small. Hence $\Sigma_0=\Sigma$, and a calculation shows that:
$$
\dfrac{d}{dt}|_{t=0}{\rm vol}_f(\Sigma_t)=-\int_{\Sigma}(H+\derive fN)ue^{-f}\,d\mu.
$$
Here $H={\tr A}$, and $A$ is the second fundamental form (shape operator) of the immersion, defined as $AX=-\nabla_XN$ on all tangent vectors $X$ to $\Sigma$;  moreover,  $\nabla$ is the Levi-Civita connection of the ambient manifold $M$.
The conclusion is that $\Sigma$ is critical for the $f$-volume if and only if 
\begin{equation}\label{fminimal}
H+\derive{f}{N}=0.
\end{equation}
If one defines the {\it $f$-mean curvature} of $\Sigma$ as 
$$
H_f\doteq H+\derive{f}{N},
$$
then $\Sigma$ is critical for the $f$-volume if and only if it is {\it $f$-minimal} in $M$, which means, by definition, that $H_f=0$ identically on $\Sigma$.

\smallskip

If $\Sigma$ is a local minimum of the weighted area functional we say that it is {\it $f$-stable}. When $\Sigma$ is unstable, it makes sense to investigate its Morse index:  this is, roughly speaking, the maximal dimension of a linear space of deformations that decrease the weighted volume up to second order. To compute it, we need to compute the second variation of the $f$-volume, which is intimately connected with the structure of the triple $M_{f}\doteq (M^{m+1},g, e^{-f}d\mu)$; such structure  is often termed a {\it weighted manifold}. 
In fact, if $uN$ is the normal variation used in \eqref{opv}, and $\Sigma$ is $f$-minimal, then:
$$
Q_f(u,u)\doteq\dfrac{d^2}{dt^2}|_{t=0}{\rm vol}_f(\Sigma_t)=
\int_{\Sigma}\Big(\abs{\nabla u}^2-({\rm Ric}_f(N,N)+\abs{A}^2)u^2\Big)e^{-f}\,d\mu
$$
where ${\rm Ric}^{M}_f=\mathrm{Ric}^{M}+\mathrm{Hess}f$ is the Bakry-\'Emery Ricci tensor of the weighted ambient manifold $M_{f}$.

\smallskip

The maximum dimension of a subspace of $C^{\infty}(\Sigma)$ on which $Q$ is negative definite is called the {\it $f$-index} of $\Sigma$, and is denoted by ${\rm Ind}_f(\Sigma)$.  This number can also be seen as the number of negative eigenvalues of the Jacobi operator $L_f$, associated to the quadratic form $Q_f$ and acting on $L^2(\Sigma, e^{-f}\,d\mu)$. We will give an explicit expression of $L_f$ later. 

\smallskip

If $\Sigma$ is compact its $f$-index is always finite. If $\Sigma$ is complete, not compact, we can define the index ${\rm Ind}_f(\Omega)$ of any relatively compact domain $\Omega\subseteq\Sigma$ as the maximum dimension of a subspace of $C^{\infty}_0(\Omega)$ (smooth functions with support in $\Omega$) on which $Q_f$ is negative definite; then define
$$
{\rm Ind}_f(\Sigma)=\sup\{{\rm Ind}_f(\Omega):\Omega\subset\subset\Sigma\}.
$$
Obviously, ${\rm Ind}_f(\Sigma)$ can be infinite. 

\smallskip

In this paper we will give lower estimates of the $f$-index of $f$-minimal hypersurfaces of the weighted manifold $(\real{m+1},g_{\rm can}, e^{-f}\,d\mu)$. When $\Sigma$ is compact and the weight $f$ is a convex function on $\real{m+1}$ the lower bound will only depend on the topology of $\Sigma$ through its first Betti number $b_1(\Sigma)$.

\begin{remark} 
\rm{The $f$-minimal equation \eqref{fminimal}, together with the rules of conformal change,  tells us that $\Sigma^m$ is $f$-minimal in $(M^{m+1},g)$ if and only if it is minimal (in the usual sense) in the manifold $(M^{m+1},e^{-\frac{2f}{m}}g)$. Moreover, the $f$-index coincides with the usual index of the minimal immersion $\Sigma^m\to (M^{m+1},e^{-\frac{2f}{m}}g)$.} 
\end{remark}

\subsection{Self-shrinkers} Perhaps the main motivation for this paper was to study an important class of $f$-minimal hypersurfaces : the {\it self-shrinkers} of the mean curvature flow. By definition, they are  connected, orientable, isometrically immersed hypersurfaces $x:\Sigma^m\to\real{m+1}$ whose mean curvature vector field $\mathbf{H}$ satisfies the equation
\begin{equation}\label{SS}
x^{\bot}=-\mathbf{H},
\end{equation}
where $(\cdot)^{\bot}$ denotes the projection on the normal bundle of $\Sigma$. Self-shrinkers play an important role in the study of singularities developed along the mean curvature flow and have been extensively studied in recent years; see e.g. \cite{CM}, \cite{CaoLi}, \cite{DX}, and references therein.

Taking the scalar product on both sides of \eqref{SS} with a unit normal vector field $N$ we see that  self-shrinkers are $f$-minimal hypersurfaces of ($\real{m+1}, g_{\rm can}, e^{-f}\,d\mu)$ for the weight function $f(x)=\frac{|x|^2}{2}$. 

\smallskip

Let us recall the main results regarding their index.
It was proved by T. Colding and W. Minicozzi, \cite{CM}, that every complete properly immersed self-shrinker is necessarily $f$-unstable (i.e it has $f$-index greater than or equal to one). Note that this result was later generalized in \cite{IR} by the first two authors, to self-shrinkers with at most exponential (intrinsic) weighted volume growth.

In the equality case, rigidity results have been proved by C. Hussey, \cite{Hus}, under the additional assumption of embeddedness. This last assumption was later removed in \cite{Imp}. More precisely, one has that if a complete properly immersed self-shrinker in $\real{m+1}$ has Morse index $1$, then it has to be an hyperplane through the origin. Furthermore, if the self-shrinker is not an hyperplane through the origin, then the Morse index has to be at least $m+2$, with equality if and only if the self-shrinker is a generalized cylinder of the form $\real{m-k}\times\sphere{k}_{\sqrt{k}}$ for some $1\leq k\leq m$. 

For other basic concepts and results about self-shrinkers and their stability properties we refer to \cite{CM} and the very recent \cite{Imp}.


\subsection{Main results} It is by now a well-established guiding principle that the index of a compact minimal hypersurface of a positively curved manifold is sensitive to the topology, more precisely, to the first Betti number of the hypersurface; hence, rich cohomology in degree one often implies high instability of the immersion. 
In fact,  Schoen, Marques and Neves conjecture that the index of a compact minimal hypersurface of a manifold with positive Ricci curvature is bounded below by an affine function of the first Betti number. 

\smallskip

In dimension $2$, this was first shown by A. Ros in \cite{Ros} for immersions in $\real 3$ or a quotient of it by a group of translations; in higher dimensions  the third author  proved this fact when the ambient manifold is the round sphere (see \cite{Sav}). Recently L. Ambrozio, A. Carlotto and B. Sharp generalized the methods of Ros and Savo and verified the conjecture for a larger class of ambient spaces (see \cite{ACS}).  However, as of today, the full conjecture above is still open, although the method was employed to obtain  bounds for minimal free boundary immersions (\cite{ACS2} and \cite{Sar}) and complete minimal immersions in $\real n$ (\cite{Li}).

\smallskip

Test-functions for the Jacobi operator are constructed as follows.  One needs the family of harmonic one-forms, which are known by classical Hodge theory to represent cohomology in degree one, and a distinguished family $\Cal P$ of vector fields on $\Sigma$, given by orthogonal projection onto $\Sigma$ of parallel vector fields in Euclidean space of suitable dimension. In dimension $2$, Ros used the family of test-functions:
$
u=\omega(V),
$
where $\omega$ is a harmonic one-form and $V\in{\Cal P}$. In higher dimensions these test-functions no longer work; indeed,  for immersions  $\Sigma^m\to\sphere{m+1}$ the third author introduced in \cite{Sav} the test-functions
$
u=\omega(X_{V,W}),
$
where $X_{V,W}=\scal{\bar V}{N}W-\scal{\bar W}{N}V$;  here $\bar V,\bar W$ are parallel vector fields in $\real{m+2}$ and $V,W$ are their projections on $\Sigma$.

\smallskip

In this paper we employ the method in \cite{Sav} to prove lower bounds in the weighted case.
We start with compact $f$-minimal hypersurfaces; in this case we can improve the bound by adding the number of small eigenvalues of the weighted Laplacian. 

Recall that the weighted Laplacian is the operator acting on $u\in C^{\infty}(\Sigma)$ as follows:
$$
\Delta_fu=\Delta u+\scal{\nabla f}{\nabla u}
$$
where $\Delta\doteq-\mathrm{div}(\nabla\,)$ is the usual Laplacian and $f$ is the weight. The weighted laplacian is self-adjoint with respect to the measure $e^{-f}\,d\mu$ and has a discrete spectrum $0=\lambda_0(\Delta_f)<\lambda_1(\Delta_f)\leq\lambda_2(\Delta_f)\leq \dots$. For any positive number $a$, set
$$
N_{\Delta_f}(a)=\#\{\text{positive eigenvalues of $\Delta_f$ which are less than $a$}\}.
$$
Here is our first main result.

\begin{theoremA}\label{maincompact}Let $\Sigma$ be a compact $f$-minimal hypersurface of the weighted manifold $M_f\doteq (\real{m+1},g_{\rm can}, e^{-f}\,d\mu)$. Assume that the Bakry-\'Emery Ricci tensor of $M_f$ is bounded below by $\mu>0$, that is,  ${\rm Ric}_f\geq\mu>0$. Then:
\begin{equation}\label{betti}
{\rm Ind}_f(\Sigma)\geq\dfrac{2}{m(m+1)}\Big({N_{\Delta_f}(2\mu)+b_1(\Sigma)}\Big).
\end{equation}
\end{theoremA}

\begin{remark}
\rm{The theorem could also be rephrased as follows: if $\Sigma$ is a compact minimal hypersurface of the Riemannian manifold $(\real{m+1},e^{-\frac{2f}{m}}g_{\rm can})$ such that ${\rm Hess}f\geq \mu>0$, then the lower bound \eqref{betti} holds.}
\end{remark}  

\smallskip

The theorem will be proved in Subsection \ref{Subsec2.2} below. In particular, the $f$-index is large provided that:

\nero the first Betti number is large, or

\nero there are many small (i.e., less than $2\mu$) eigenvalues of the weighted Laplacian. 

\smallskip

When $\Sigma$ is a self-shrinker we take $f(x)=\frac{\abs{x}^2}{2}$. In that case ${\rm Hess}_f={\rm Ric}^{M}_f=g_{\rm can}$ and we get an affine lower bound.

\begin{corollaryA}\label{MainCorCompactSS}
Let $x:\Sigma^{m}\to\real{m+1}$ be a compact self-shrinker. Then
\begin{equation}\label{IndEst}
\mathrm{Ind}_{f}(\Sigma)\geq\frac{2}{m(m+1)}{b_{1}(\Sigma)}+m+1.
\end{equation}
In particular if $\Sigma$ has dimension $m=2$, letting $g=\mathrm{genus}(\Sigma)$, we have that
$$
\ \mathrm{Ind}_{f}(\Sigma)\geq \frac{2}{3}g+3.
$$
\end{corollaryA}

It should be said that self-shrinkers behave, in some respects, like closed minimal surfaces of $\sphere 3$. We remark that in that case we have (see \cite{Sav}):
$$
{\rm Ind}(\Sigma)\geq \dfrac g2+4.
$$

We adopt the method in \cite{Sav}, with a noteworthy modification. Instead of using the usual harmonic one-forms we employ the so-called {\it $f$-harmonic} one forms, which satisfy by definition the two conditions:
$$
d\omega=\delta_f\omega=0
$$
where $\delta_f\omega=\delta\omega+\omega(\nabla f)$ is the weighted  codifferential. This seems to be the natural approach in the weighted case, also because $f$-harmonic forms are those which minimize the functional
$$
\omega\mapsto \int_{\Sigma}\abs{\omega}^2e^{-f}\,d\mu
$$
restricted to a fixed cohomology class. By the Hodge decomposition (which continues to hold in the weighted case, see \cite{Bue}) the dimension of the space of harmonic $1$-forms equals the first Betti number of $\Sigma$. 

It should be pointed out that P. Zhu and W. Gan  also obtained an estimate in the spirit of the Theorem \ref{maincompact} above: a lower bound by the first Betti number is proved, but only assuming  an additional curvature condition (see Corollary 1.2 in \cite{ZG}). In that paper usual harmonic (and not $f$-harmonic) forms are being used. 
\bigskip

Adapting to this setting the approach used in \cite{ACS} we obtained similar estimates also in the case of compact $f$-minimal hypersurfaces in more general ambient weighted manifolds. The article \cite{IRS_II} containing these results will be available soon.
\bigskip

In the second part of the paper we consider the case where the immersion is complete, non-compact. Denote by ${\Cal H}^1_f(\Sigma)$ the space of $f$-harmonic one-forms $\omega$ which are square summable for the weighted measure:
$$
{\Cal H}^1_f(\Sigma)=\left\{\omega\in\Lambda^1(\Sigma): d\omega=\delta_f\omega=0, \,\, \int_{\Sigma}\abs{\omega}^2e^{-f}\,d\mu<+\infty\right\}.
$$

\begin{theoremA} Let $\Sigma$ be a complete, non-compact $f$-minimal hypersurface
of $(\real{m+1},g_{\rm can},e^{-f}\,d\mu)$ with ${\rm Ric}_f\geq\mu>0$. 
Then
$$
{\rm Ind}_f(\Sigma)\geq \dfrac{2}{m(m+1)}{\rm dim}{\Cal H}^1_f(\Sigma).
$$
In particular, if $\Sigma^m$ is a complete properly immersed self-shrinker, then
$$
{\rm Ind}_f(\Sigma)\geq \dfrac{2}{m(m+1)}\Big(m+1+{\rm dim}{\Cal H}^1_f(\Sigma)\Big).
$$
\end{theoremA}
Adapting the well-known Farkas-Kra construction (see \cite{FK}) to the weighted situation we will prove in the last section:

\begin{theoremA}\label{twog} Let $\Sigma$ be a two-dimensional orientable, connected, complete surface. Then, for all $f\in C^{\infty}(\Sigma)$:
$$
{\rm dim}{\Cal H}^1_f(\Sigma)\geq 2g
$$
where $g$ is the genus of $\Sigma$. 
\end{theoremA}This implies the following inequalities.

\begin{corollaryA}\label{CorE} \item[ (a)] Let $\Sigma$ be a complete, non-compact $f$-minimal surface in $(\real{3},g_{\rm can},e^{-f}\,d\mu)$ with ${\rm Ric}_f\geq\mu>0$. Then
$$
{\rm Ind}_f(\Sigma)\geq \dfrac{2g}3.
$$
In particular, stable $f$-minimal surfaces have genus zero. 

\item[(b)] If $\Sigma^2$ is a complete self-shrinker, properly immersed in $\real 3$, then
$$
{\rm Ind}_f(\Sigma)\geq \dfrac{2g}3+1
$$
\end{corollaryA}

Ros proved the bound in (a) for usual complete minimal surfaces in $\real 3$ (see \cite{Ros}); this was later improved to ${\rm Ind}(\Sigma)\geq \frac23(g+r)-1$ by O. Chodosh and D. Maximo (\cite{ChMa}), where $r$ is the number of ends. 
\smallskip
The inequality in b)  improves the lower bound ${\rm Ind}_f(\Sigma)\geq \frac g3$  proved in \cite{McG2} under the additional condition that ${\rm sup}_{\Sigma}\abs{k_1^2-k_2^2}\leq\delta<1$ (here $k_1$ and $k_2$ are the principal curvatures of $\Sigma$). 


\section{A comparison theorem in the compact case} In the compact case, our estimates will be a consequence of a more general comparison result between the spectrum of the stability operator and that of the $f$-Laplacian acting on $1$-forms. 
Thus, we start by defining these operators.

\subsection{Hodge Laplacian and Weitzenb\"ock formula in the weighted setting}

Recall that the $f$-Laplacian of a weighted manifold $(M,g,e^{-f}\,d\mu)$, acting on functions,  is defined by:
$$
\Delta_fu=\Delta u+\scal{\nabla f}{\nabla u}.
$$
In general, if one introduces the weighted divergence
$$
\delta_f=\delta+i_{\nabla f},
$$
then one has simply $\Delta_fu=\delta_f du$. Note that $\Delta_f$ is self-adjoint with respect to the weighted measure $e^{-f}d\mu$. More generally, we have a Hodge $f$-Laplacian acting on $p$-forms, denoted $\Delta_f^{[p]}$ and  defined in the natural way:
$$
\Delta^{[p]}_f=d\delta_f+\delta_f d.
$$
As $M$ is compact, the Hodge Laplacian has a discrete spectrum $\{\lambda_k(\Delta^{[p]}_f)\}_{k=1,2,\dots}$. The important fact is that the Hodge decomposition continues to hold in this setting; therefore the dimension of the kernel of $\Delta^{[p]}_f$ is equal to the $p$-th Betti number, which means that
$$
\lambda_k(\Delta^{[p]}_f)=0\quad\text{for $k=1,\dots,b_p(M)$}.
$$
The following Lemma is well-known in the field; it gives an expression  of the  Hodge $f$-Laplacian in terms of the connection Laplacian and the Bakry-\'Emery Ricci tensor. Since we haven't found a proof in literature, we provide it below.
\begin{lemma}\label{LemfWeitz}
 Let $(M, g, e^{-f}d\mu)$ be a weighted manifold and $\omega\in C^{\infty}(\Lambda^{1}(M))$. Then
\begin{equation}\label{f-Weitz}
\Delta_{f}^{[1]}\omega=\nabla_{f}^{*}\nabla\omega+\mathrm{Ric}_{f}(\omega^{\sharp}),
\end{equation}
where $\Delta_{f}^{[1]}=\delta_{f}d+d\delta_{f}$, $\delta_{f}=\delta+i_{\nabla f}$, $\nabla_{f}^{*}=\nabla^{*}+i_{\nabla f}$, $\mathrm{Ric}_{f}=\mathrm{Ric}+\mathrm{Hess}f$.
\end{lemma}

\begin{proof}
Recall that by the classical Weitzenb\"ock formula, letting $\omega\in C^{\infty}(\Lambda^{1}(M))$, we have that
\[
\ \Delta^{[1]}\omega=\nabla^{*}\nabla\omega+\mathrm{Ric}(\omega^{\sharp}).
\]
Then
\begin{eqnarray*}
\Delta_{f}^{[1]}\omega&=&\left(\delta_{f}d+d\delta_{f}\right)\omega= \left(\delta+i_{\nabla f}d+d\delta+d\,i_{\nabla f}\right)\omega\\
&=&(\Delta^{[1]}+\mathcal{L}_{\nabla f})\omega=\nabla^{*}\nabla\omega+\mathrm{Ric}(\omega^{\sharp})+\mathcal{L}_{\nabla f}\omega\\
&=&\nabla_{f}^{*}\nabla \omega+\mathrm{Ric}(\omega^{\sharp})-i_{\nabla f}\nabla \omega+\mathcal{L}_{\nabla f}\omega,
\end{eqnarray*}
where of course $\mathcal L$ is the Lie derivative. Since, for every $X\in TM$,
\begin{eqnarray*}
-(i_{\nabla f}\nabla \omega)(X)+\left(\mathcal{L}_{\nabla f}\omega\right)(X)&=&-\nabla\omega(\nabla f, X)+(i_{\nabla f}d\omega)(X)+d(i_{\nabla f}\omega)(X)\\
&=&-\nabla\omega(\nabla f,X)+d\omega(\nabla f,X)+X(\omega(\nabla f))\\
&=& - \nabla\omega(X,\nabla f)+\nabla_X\omega(\nabla f)+\omega(\nabla_X\nabla f)\\
&=& \scal{\nabla_X\nabla f}{\omega^{\sharp}}
\end{eqnarray*}
and since $\scal{\nabla_X\nabla f}{\omega^{\sharp}}={\rm Hess}f(X,\omega^{\sharp})={\rm Hess}f(\omega^{\sharp},X)$ we see from the two previous facts that 
\begin{eqnarray*}
\Delta_{f}^{[1]}\omega(X)&=&\nabla_{f}^{*}\nabla \omega(X)+\mathrm{Ric}(\omega^{\sharp},X)+{\rm Hess}f(\omega^{\sharp},X)\\
&=&\nabla_{f}^{*}\nabla \omega(X)+\mathrm{Ric}_f(\omega^{\sharp},X)
\end{eqnarray*}
which is the assertion. 
\end{proof}

For notational purposes we will prefer to work with vector fields instead of one-forms. If $\xi$ is a vector field on $M$ consider its dual one-form $\xi^{\flat}$.
We define the Hodge $f$-Laplacian of $\xi$ as the unique vector field such that
$$
\scal{\Delta^{[1]}_f\xi}{X}=\Delta^{[1]}_f\xi^{\flat}(X)
$$
for all $X\in TM$.
The Weitzenb\"ock formula becomes:
$$
\Delta^{[1]}_f\xi=\nabla^{*}_f\nabla\xi+\mathrm{Ric}_f(\xi),
$$
where
$
\nabla_{f}^{*}\nabla\xi=\nabla^{*}\nabla\xi+\nabla_{\nabla f} \xi,
$
and, with respect to an orthonormal basis, 
\[
\nabla^{*}\nabla\xi=-\sum_{i}(\nabla_{e_i}\nabla_{e_i}\xi-\nabla_{\nabla_{e_i}e_i}\xi).
\]

\subsection{A comparison theorem and the proof of Theorem \ref{maincompact} and Corollary \ref{MainCorCompactSS}}\label{Subsec2.2}
Now let $\Sigma$ be a complete $f$-minimal hypersurface of the weighted manifold $M_{f}\doteq(\real{m+1},g_{\rm can}, e^{-f}d\mu)$.  As such, it inherits a structure of weighted manifold, the weight being simply the restriction of $f$ to $\Sigma$ which, by a slight abuse of language, we keep denoting by the same letter $f$. The stability operator of $\Sigma$ is then given by
$$
L_fu=\Delta_fu-({\rm Ric}^M_f(N,N)+\abs{A}^2)u,
$$
where $\Delta_f$ is the weighted Laplacian of $\Sigma$ and ${\rm Ric}^M_f$ is the Bakry-\'Emery Ricci tensor of the ambient weighted manifold; of course one has
$$
{\rm Ric}_{f}^M={\rm Hess}f,
$$
(the hessian of $f$, computed in $\real{m+1}$).
Similar to the case of minimal immersions in the standard sphere, we have the following comparison theorem between the spectrum of the stability operator and that of the Hodge $f$-Laplacian acting on $1$-forms. 

\begin{theorem}\label{MainThmCompactfMin} Let $M_{f}=(\real{m+1}, g_{\rm can}, e^{-f}d\mu)$ be a weighted manifold such that ${\rm Ric}^M_f\doteq {\rm Hess}f\geq \mu>0$.
Let $\Sigma^m\to\real{m+1}$ be a compact $f$-minimal hypersurface. Then, for all $k$:
$$
\lambda_k(L_f)\leq -2\mu+\lambda_{d(k)}(\Delta_f^{[1]})
$$
where $\Delta_f^{[1]}=\delta_{f}d+d\delta_{f}$ is the Hodge $f$-Laplacian acting on $1$-forms of $\Sigma$  and 
$$
d(k)=\binom{m+1}{2}(k-1)+1.
$$
\end{theorem}

We will prove the theorem in the next section. 

Now let
$$
\beta=\#\{\text{eigenvalues of $\Delta^{[1]}_f$ which are less than $2\mu$}\}.
$$
If $k$ is the largest integer such that $d(k)\leq\beta$, one sees from the above inequality that $\lambda_k(L_f)<0$, so that ${\rm Ind}_f(\Sigma)\geq k$. It is easy to estimate that $k\geq \frac{2}{m(m+1)}\beta$. Then the theorem gives:
\begin{equation}\label{beta}
{\rm Ind}_f(\Sigma)\geq \dfrac{2}{m(m+1)}\beta.
\end{equation}
{\it Proof of  Theorem \ref{maincompact}}. Let $\gamma=N_{\Delta_f}(2\mu)$ be the number of positive eigenvalues of the $f$-Laplacian which are less than $2\mu$. 
Let $u_1,\dots,u_{\gamma}$ be $L^2_f$-orthogonal eigenfunctions of $\Delta_f$ associated to positive eigenvalues which are less than $2\mu$. Obviously, they are all orthogonal to constants. As $\Delta^{[1]}_f$ commutes with exterior differentiation $d$, the $1$-forms
$du_1,\dots,du_{\gamma}$ form, by Stokes formula,  an orthogonal set of eigenforms  of $\Delta^{[1]}_f$ associated to positive eigenvalues less than $2\mu$. As they are all orthogonal to the space of $f$-harmonic $1$-forms, we see that
$
\beta\geq \gamma+b_1(\Sigma),
$
hence, from \eqref{beta}:
$$
{\rm Ind}_f(\Sigma)\geq \dfrac{2}{m(m+1)}\Big(N_{\Delta_f}(2\mu)+b_1(\Sigma)\Big)
$$
which is the statement of Theorem \ref{maincompact}.

\medskip

{\it Proof of Corollary \ref{MainCorCompactSS}}. Reasoning as before, now let $k=\frac{2}{m(m+1)}b_1(\Sigma)$. Then $d(k)\leq b_1(\Sigma)$ and we see from the theorem that $\lambda_k(L_f)\leq -2$. This means that there are at least $k$ eigenvalues of $L_f$ which are less than or equal to $-2$. Note now that by \eqref{BasEq3} below, for any self-shrinker, we have at least $m+1$ eigenvalues equal to $-1$. Therefore, the index is at least $k+m+1$ and the assertion follows.


\section{Proof of Theorem \ref{MainThmCompactfMin}}

\subsection{The test-functions and the main computational lemma}
Then let $\Sigma$ be an hypersurface of $(\real{m+1}, g_{\rm can},e^{-f}\,d\mu)$. Here and in the rest of the paper, we will denote by $\overline{\mathcal{P}}$ the set of parallel vector fields on $\real{m+1}$, and by $V$ the orthogonal projection of the parallel field $\bar V\in \overline{\mathcal{P}}$ on $\Sigma$, so that we have
$$
\bar V=V+\scal{\bar V}{N}N.
$$
To a pair of parallel vector fields $\bar V,\bar W\in \overline{\mathcal P}$ we associate the vector field on $\Sigma$ defined by $X_{\bar{V},\bar{W}}=\<\bar{V},N\>W-\<\bar{W},N\>V$. We obtain a family of test functions for the weighted Jacobi operator by pairing $X_{\bar{V},\bar{W}}$ with a vector field $\xi\in T\Sigma$:
\begin{equation}\label{tf}
u=\left\langle X_{\bar{V},\bar{W}},\xi\right\rangle.
\end{equation}
Typically,  $\xi$ will be a $f$-harmonic vector field, or an eigenvector field of the Hodge $f$-Laplacian. The scope of the following Lemma is to give an explicit expression of the stability operator when applied to $u$.

\begin{lemma}\label{Lem3} Let $f\in C^{\infty}(\real{m+1})$ and let $x:\Sigma^{m}\to\real{m+1}$ be an $f$-minimal hypersurface.
Let $\xi\in T\Sigma$ be a generic vector field on $\Sigma$ and $u$ the function defined in \eqref{tf}. Then
\[
\ L_{f} u=-u\mathrm{Hess}f(N,N)-\mathrm{Hess}f(X_{\bar{V},\bar{W}},\xi)+\<X_{\bar{V},\bar{W}},\Delta_{f}^{[1]}\xi\>+v,
\]
where $v=2(\<\nabla_{AV}\xi,W\>-\<\nabla_{AW}\xi, V\>)-\<W,\xi\>\mathrm{Hess}f(V,N)+\<V,\xi\>\mathrm{Hess}f(W,N)$.

\medskip

If $\Sigma$ is a self-shrinker (so that ${\rm Hess}f=g_{\rm can}$) then
$$
L_fu=-2u+\<X_{\bar{V},\bar{W}},\Delta_{f}^{[1]}\xi\>+v,
$$
where
$
v=2(\<\nabla_{AV}\xi,W\>-\<\nabla_{AW}\xi, V\>).
$
\end{lemma}

\begin{proof} Recall that 
$$
L_fu=\Delta_fu-({\rm Ric}^M_f(N,N)+\abs{A}^2)u=\Delta_fu-({\rm Hess} f(N,N)+\abs{A}^2)u.
$$ 
The assertion now follows from \eqref{deltafu}, which will be proved in the next section. 
\end{proof}


\subsection{The basic equations and the proof of Lemma \ref{Lem3}}
In the next two lemmas we collect some preliminary computations which will be used in the proof of our main results. 
\begin{lemma}\label{LemParVectFields}
Let $f\in C^{\infty}(\real{m+1})$ and let $x:\Sigma^{m}\to\real{m+1}$ be an $f$-minimal hypersurface. Denote by $N$ the unit normal vector and by $A$ the second fundamental form of $\Sigma$. Let $\bar{V}\in\overline{\mathcal{P}}$ and denote by  $V$ its projection on $\Sigma$. If $X\in T\Sigma$  one has
\begin{eqnarray}
\nabla_{X}V&=&\left\langle\bar{V},N\right\rangle AX;\label{BasEq1}\\
\nabla\left\langle\bar{V},N\right\rangle&=&-AV;\label{BasEq2}\\
\nabla_{f}^{*}\nabla V&=&A^2 V+ \<\bar{V},N\>\mathrm{Hess}f(N)^T.\label{BasEqV}
\end{eqnarray}
\end{lemma}
\begin{proof}
Let $\bar\nabla$ denote the Levi-Civita connection on $\real{m+1}$. We have that
\begin{align*}
\nabla_{X}V&=\left(\overline{\nabla}_{X}V\right)^{T}=-\left(\overline{\nabla}_{X}\left\langle\bar{V}, N\right\rangle N\right)^{T}=-\left\langle\bar{V}, N\right\rangle\left(\overline{\nabla}_{X}N\right)\\
&=\left\langle\bar{V}, N\right\rangle AX,
\end{align*}
and, for all tangent vectors $X$:
\begin{equation*}
\left\langle\nabla\left\langle\bar{V},N\right\rangle,X\right\rangle=X\left\langle\bar{V},N\right\rangle=\left\langle\bar{V},\overline{\nabla}_{X}N\right\rangle=-\left\langle AV,X\right\rangle.
\end{equation*}
As for \eqref{BasEqV}, note that $\nabla_{e_i}V=\scal{\bar V}{N}Ae_i$. Then:
\begin{align*}
\nabla^*\nabla V&=-\sum_i(\nabla_{e_i}(\<\bar{V},N\>Ae_i)-\<\bar{V},N\>A{\nabla_{e_i}}e_i)\\
&=-A\nabla \<\bar{V},N\>-\<\bar{V},N\>\mathrm{tr}\nabla A\\
&=A^2 V- \<\bar{V},N\>\nabla H.
\end{align*}
Moreover, using the $f$-minimal equation, it is not difficult to show that
\begin{equation}\label{NablaH}
\nabla H=-\mathrm{Hess}f(N)^T+A\nabla f.
\end{equation}
Hence we get
\begin{align*}
\nabla^*\nabla V&=A^2 V+ \<\bar{V},N\>\mathrm{Hess}f(N)^T- \<\bar{V},N\>A\nabla f\\
&=A^2 V+ \<\bar{V},N\>\mathrm{Hess}f(N)^T- \nabla_{\nabla f} V.
\end{align*}
\end{proof}

\begin{lemma}\label{Lem2}
Let $f\in C^{\infty}(\real{m+1})$ and let $x:\Sigma^{m}\to\real{m+1}$ be an $f$-minimal hypersurface. Let $\bar{V},\,\bar{W}\in\overline{\mathcal{P}}$, $V,\,W$ their projections on $\Sigma$. Then, for any $\xi\in T\Sigma$
\begin{align}
\Delta_{f}\left\langle\bar{V},N\right\rangle=&|A|^2\left\langle\bar{V},N\right\rangle-\mathrm{Hess}f(V,N);\label{BasEq3}\\
\Delta_{f}\left\langle V,\xi\right\rangle=&-\mathrm{Hess}f(V,\xi)+\langle \bar{V},N\rangle\mathrm{Hess}f(N,\xi)+2\left\langle AV,A\xi\right\rangle\label{BasEq4}\\
&-2\left\langle\bar{V},N\right\rangle\left\langle\nabla\xi,A\right\rangle+\left\langle\Delta_{f}^{[1]}\xi,V\right\rangle;\nonumber\\
\Delta_{f}\left(\left\langle\bar{V},N\right\rangle\left\langle W,\xi\right\rangle\right)=&|A|^2\<\bar{V},N\>\<W,\xi\>+2\scal{\bar V}{N}\scal{AW}{A\xi}+2\scal{\bar W}{N}\scal{AV}{A\xi}\label{BasEq5}\\
&-\<\bar{V},N\>\mathrm{Hess}f(W,\xi)-\<W,\xi\>\mathrm{Hess}f(V,N)\nonumber\\
&+\langle\bar{V},N\rangle\langle \bar{W},N\rangle\mathrm{Hess}f(N,\xi)+2\nabla\xi(AV,W)\nonumber\\
&-2\<\bar{V},N\>\<\bar{W},N\>\<\nabla\xi,A\>+\<\bar{V},N\>\<\Delta_{f}^{[1]}\xi,W \>.\nonumber
\end{align}
Finally, if $X_{\bar V,\bar W}=\scal{\bar V}{N}W-\scal{\bar W}{N}V$  and 
$u=\scal{X_{\bar V,\bar W}}{\xi}$ then:
\begin{equation}\label{deltafu}
\Delta_f u=\abs{A}^2u-\mathrm{Hess}f(X_{\bar{V},\bar{W}},\xi)+\<X_{\bar{V},\bar{W}},\Delta_{f}^{[1]}\xi\>+v,
\end{equation}
where $v=2(\<\nabla_{AV}\xi,W\>-\<\nabla_{AW}\xi, V\>)-\<W,\xi\>\mathrm{Hess}f(V,N)+\<V,\xi\>\mathrm{Hess}f(W,N)$.
\end{lemma}

\begin{proof}
First note that, as a consequence of the Codazzi equation, we have that
\[
\mathrm{div}(AV)=\<\bar{V}, N\>|A|^2+\<\nabla H, V\>.
\]
Therefore we obtain, by \eqref{NablaH} and \eqref{BasEq2}:
\begin{align*}
\Delta \< \bar{V},N\>&=-\mathrm{div}(\nabla \<\bar{V}, N\>)=\mathrm{div}(AV)\\
&=\<\bar{V}, N\>|A|^2-\<\mathrm{Hess}f(N), V\>+\<A\nabla f, V\>\\
&=\<\bar{V}, N\>|A|^2-\mathrm{Hess}f(N,V)-\<\nabla f, \nabla \< \bar{V},N\>\>.
\end{align*}
Equation \eqref{BasEq3} now follows by the definition of $\Delta_f$:
\[
\Delta_f \scal{\bar V}{N}=\Delta \scal{\bar V}{N}+\<\nabla f,\nabla \scal{\bar V}{N}\>.
\]

As for equation \eqref{BasEq4}, observe that by Lemma \ref{LemfWeitz} and \eqref{BasEqV}
\begin{eqnarray*}
\Delta_{f}\<V,\xi\>&=&\<\nabla_{f}^{*}\nabla V, \xi\>+\<V, \nabla^{*}_{f}\nabla\xi\>-2\<\nabla V,\nabla\xi\>\\
&=&\<A V,A\xi\>+\<\bar{V},N\>\mathrm{Hess}f(N,\xi)+\<\Delta_{f}^{[1]}\xi, V\>-\mathrm{Ric}_{f}^{\Sigma}(V, \xi)-2\<\nabla V,\nabla\xi\>.
\end{eqnarray*}

Moreover (see e.g. \cite{IR}) for an $f$-minimal hypersurface in the Euclidean space we have that
\[
\ \mathrm{Ric}_{f}^{\Sigma}(\xi)=\mathrm{Hess}f(\xi)-A^2(\xi),\quad \xi\in T\Sigma.
\]
We hence get
\begin{equation*}
\Delta_{f}\<V,\xi\>=-\mathrm{Hess}f(V,\xi)+\<\bar{V},N\>\mathrm{Hess}f(N,\xi)+2\<AV,A\xi\>+\<\Delta_{f}^{[1]}\xi,V\>-2\<\nabla V,\nabla\xi\>.
\end{equation*}
Note also that, by \eqref{BasEq1}, 
\[
\ \<\nabla V,\nabla\xi\>=\<\bar{V},N\>\<A,\nabla\xi\>,
\]
Thus
\[
\ \Delta_{f}\<V,\xi\>=-\mathrm{Hess}f(V,\xi)+\<\bar{V},N\>\mathrm{Hess}f(N,\xi)+2\<AV,A\xi\>+\<\Delta_{f}^{[1]}\xi,V\>-2\<\bar{V},N\>\<A\,,\nabla\xi\,\>.
\]
We now note that, for any $X\in T\Sigma$:
$$
\begin{aligned}
\scal{\nabla\scal{W}{\xi}}{X}=&\scal{\nabla_XW}{\xi}+\scal{W}{\nabla_X\xi}\\
=&\scal{\bar W}{N}\scal{A\xi}{X}+\nabla\xi(X,W)
\end{aligned}
$$
so that
\begin{equation}\label{dn}
\begin{aligned}
\scal{\nabla\scal{\bar V}{N}}{\nabla\scal{W}{\xi}}&=-\scal{AV}{\nabla\scal{W}{\xi}}\\
&=-\scal{\bar W}{N}\scal{AV}{A\xi}-\nabla\xi(AV,W)
\end{aligned}
\end{equation}
As
\begin{equation}
\Delta_f\left(\scal{\bar V}{N}\scal{W}{\xi}\right)=\scal{W}{\xi}\Delta_f\scal{\bar V}{N}+\scal{\bar V}{N}\Delta_f\scal{W}{\xi}-2\scal{\nabla\scal{\bar V}{N}}{\nabla\scal{W}{\xi}}
\end{equation}
equation \eqref{BasEq5} now follows by substituting in the above expression \eqref{BasEq3}, \eqref{BasEq4} and \eqref{dn}. 

Finally, \eqref{deltafu} follows by using formula \eqref{BasEq5} twice. 
\end{proof}

We also observe the following fact, which will be used later. 

\begin{lemma}\label{linearfunctions} Let $x:\Sigma^m\to\real{m+1}$ be a self-shrinker and $\bar V$ be a parallel vector field on $\real{m+1}$. Then every linear function $u=\scal{\bar V}{x}$, restricted to $\Sigma$, satisfies: 
$$
\nabla\scal{\bar V}{x}=V\quad \Delta_f\scal{\bar V}{x}= \scal{\bar V}{x}.
$$
Hence any such $u$ is an eigenfunction of the weighted Laplacian associated to the eigenvalue $\lambda=1$.
Moreover, if $\Sigma$ is complete and properly immersed then $u\in W^{1,2}(\Sigma_f)$, that is:
$$
\int_{\Sigma}(u^2+\abs{\nabla u}^2)e^{-f}d\mu < +\infty.
$$
\end{lemma}

\begin{proof} The proof of the first statement is easy. If $\Sigma$ is complete and properly immersed it is known that then it is of finite volume and, more generally, every polynomial in $\abs{x}$ is $f$-integrable (for details we refer to \cite{Imp}). 
The second assertion now follows because 
 $\abs{u}\leq \abs{x}$ and $\abs{\nabla u}$ is bounded by a constant.
\end{proof}


\subsection{End of proof of Theorem \ref{MainThmCompactfMin}}

Select an orthonormal basis $\left\{\varphi_{j}\right\}$ of $L^{2}(\Sigma_{f})=L^{2}(\Sigma, e^{-f}d\mu)$ given by eigenfunctions of $L_{f}$, where $\varphi_{j}$ is associated to $\lambda_{j}(L_{f})$, and let $E^d$ be the direct sum of the first $d$ eigenspaces of $\Delta_{f}^{[1]}$:
$$
E^d=\bigoplus_{j=1}^d V_{\Delta_f^{[1]}}(\lambda_j).
$$
We look for vector fields $\xi\in E^d$ such that $u=\<X_{\bar{V},\bar{W}},\xi\>$ satisfies the following orthogonality relations  for all choices of $\bar{V},\bar{W}\in\bar{\mathcal{P}}$:
\[
\ \int_{\Sigma}\<X_{\bar{V},\bar{W}},\xi\>\varphi_{1}e^{-f}d\mu=\ldots=\int_{\Sigma}\<X_{\bar{V},\bar{W}},\xi\>\varphi_{k-1}e^{-f}d\mu
\]
As the vector space $\bar{\mathcal{P}}$ has dimension $m+1$ and since $X_{\bar{V},\bar{W}}$ is a skew symmetric bilinear function of $\bar{V},\bar{W}$, we see that the above is a system of ${{m+1}\choose{2}}(k-1)$ homogeneous linear equations in the unknown $\xi\in E^d$.

If $d=d(k)={{m+1}\choose{2}}(k-1)+1$, we can then find a non-trivial vector field $\xi\in E^d$ such that $u=\<X_{\bar{V},\bar{W}},\xi\>$ is $L^{2}(\Sigma_f)$-orthogonal to the first $k-1$ eigenfunctions of $L_{f}$ for all $\bar{V},\,\bar{W}$. Then, by the min-max principle, we have that
\begin{equation}\label{EqLambdak}
\lambda_{k}(L_{f})\int_{\Sigma}u^2e^{-f}d\mu\leq\int_\Sigma uL_{f}u\,e^{-f}d\mu.
\end{equation}
Let $\bar{\mathcal{U}}$ be the family of parallel vector fields of $\real{m+1}$ having unit length. As in \cite{Sav}, we identify $\bar{\mathcal U}$ with $\sphere m$ and endow it with the measure $\hat\mu=\frac{m+1}{\abs{\sphere m}}{\rm dvol}_{\sphere m}$.
Using coordinates, one verifies easily  that, for all $\bar{X},\bar{Y}\in\real{m+1}$:
\begin{equation}\label{EqBarU}
\int_{\bar{\mathcal{U}}}\<\bar{V},\bar{X}\>\<\bar{V},\bar{Y}\>d\bar{V}=\<\bar{X},\bar{Y}\>.
\end{equation}
Using the product metric on $\bar{\mathcal{U}}\times\bar{\mathcal{U}}$ and applying \eqref{EqBarU} repeatedly, we see that at each $x\in\Sigma$:
\begin{eqnarray*}
&&\int_{\bar{\mathcal{U}}\times\bar{\mathcal{U}}}u^{2}d\bar{V}d\bar{W}=2|\xi|^2;\\
&&\int_{\bar{\mathcal{U}}\times\bar{\mathcal{U}}}\mathrm{Hess}f(X_{\bar{V},\bar{W}},\xi)u d\bar{V}d\bar{W}=2\mathrm{Hess}f(\xi,\xi);\\
&&\int_{\bar{\mathcal{U}}\times\bar{\mathcal{U}}}\mathrm{Hess}f(N,N)u^2d\bar{V}d\bar{W}=2\mathrm{Hess}f(N,N)|\xi|^2;\\
&&\int_{\bar{\mathcal{U}}\times\bar{\mathcal{U}}}u\<X_{\bar{V},\bar{W}},\Delta_{f}^{H}\xi\>d\bar{V}d\bar{W}=2\<\xi,\Delta_{f}^{[1]}\xi\>;\\
&&\int_{\bar{\mathcal{U}}\times\bar{\mathcal{U}}}u v d\bar{V}d\bar{W}=0.
\end{eqnarray*}
Integrating \eqref{EqLambdak} with respect to $\left(\bar{V},\bar{W}\right)\in\bar{\mathcal{U}}\times\bar{\mathcal{U}}$, applying the Fubini theorem and Lemma \ref{Lem3}, one concludes that
\[
\ \lambda_{k}(L_{f})\int_{\Sigma}|\xi|^2e^{-f}d\mu\leq-\int_{\Sigma}(\mathrm{Hess}f(\xi,\xi)+\mathrm{Hess}f(N,N)|\xi|^2)e^{-f}d\mu+\int_{\Sigma}\<\xi,\Delta_{f}^{[1]}\xi\>e^{-f}d\mu.
\]
Now note that, as $\xi$ is a linear combination of the first $d(k)$ eigenvector fields of $\Delta_{f}^{[1]}$, one easily verifies that
\[
\ \int_{\Sigma}\<\xi,\Delta_{f}^{[1]}\xi\>e^{-f}d\mu\leq\lambda_{d(k)}(\Delta_{f}^{[1]})\int_{\Sigma}|\xi|^2e^{-f}d\mu.\]
Putting together the above facts with the assumption  $\mathrm{Hess}f\geq\mu$,  the assertion of the Theorem follows.

\section{Non-compact case : proofs} 

We now assume that the immersion $x:\Sigma^m\to\real{m+1}$ is  complete and non-compact, and let $f\in C^{\infty}(\real{m+1})$ be a given weight.  
We let $\mathcal H^1_f(\Sigma)$ denote the space of $f$-square summable $f$-harmonic vector fields on $\Sigma$:
$$
\mathcal H^1_f(\Sigma)=\{\xi\in T\Sigma: d\xi=\delta_f\xi=0, \int_{\Sigma}\abs{\xi}^2e^{-f}\,d\mu<+\infty\}.
$$
We want to estimate the $f$-index from below, in terms of the dimension of 
$\mathcal H^1_f(\Sigma)$. Actually, we give a slightly stronger estimate. 
We denote by $V_{\Delta_f}(\lambda)$ the space of $f$-square summable eigenfunctions of $\Delta_f$ associated to $\lambda$ and having finite weighted Dirichlet integral:
$$
V_{\Delta_f}(\lambda)=\{u\in C^{\infty}(\Sigma): \Delta_fu=\lambda u, \int_{\Sigma}(u^2+\abs{\nabla u}^2)e^{-f}d\mu <+\infty\}.
$$

For a fixed $\Lambda>0$, we  let $E(\Sigma,\Lambda)$ be the vector space generated by all vector fields which are gradients of some $u\in V_{\Delta_f}(\lambda)$ with $\lambda<\Lambda$:
$$
E(\Sigma,\Lambda)={\rm span}\{\nabla u: u\in V_{\Delta_f}(\lambda), \lambda\leq \Lambda\}.
$$
In the next section we will prove the following fact.

\begin{theorem} \label{complete} Assume that $\Sigma^m$ is a complete, non-compact $f$-minimal immersed hypersurface of the weighted space $M_{f}=(\real{m+1},g_{can},e^{-f}\,d\mu)$, such that $\Ric_f\geq \mu>0$. Assume that ${\rm Ind}_f(\Sigma)$ is finite. Then $\mathcal H^1_f(\Sigma)$ and $E(\Sigma,\Lambda)$ have finite dimensions for any $\Lambda<2\mu$, and
$$
{\rm Ind}_f(\Sigma)\geq \dfrac{2}{m(m+1)}\Big(\dim E(\Sigma,\Lambda) + \dim  \mathcal H^1_f(\Sigma)\Big).
$$
\end{theorem}

Specializing to the case $f=\frac 12\abs{x}^2$, we have the following consequence.

\begin{corollary}\label{dimensiontwo} Let $\Sigma^m$ be a complete, properly immersed self-shrinker which is not a hyperplane. Then:
$$
{\rm Ind}_f(\Sigma)\geq \dfrac{2}{m(m+1)}\Big(m+1 + \dim  \mathcal H^1_f(\Sigma)\Big).
$$
\end{corollary} 

\begin{proof} From Lemma \ref{linearfunctions} we see that each $V\in\mathcal P$ is the gradient of the linear function $u=\scal{\bar V}{x}$, and this function belongs to $V_{\Delta_f}(1)$. Then $\mathcal P\subseteq E(\Sigma,1)$. Taking $\Lambda=1$ in the theorem (note that $1<2\mu$ because for a self-shrinker $2\mu=2$), it is enough to show that
$$
{\rm dim} \mathcal P\geq m+1.
$$
Let $\mathcal B=(\bar{V}_1,\dots,\bar{V}_{m+1})$ be an orthonormal basis of $\real{m+1}$, and denote by $V_1,\dots,V_{m+1}$ the vector fields obtained by projection of $\mathcal B$ on $\Sigma$. If these vector fields were linearly dependent, we would have a parallel vector field on $\real{m+1}$ which is everywhere normal to $\Sigma$: this can't happen unless $\Sigma$ is a hyperplane. 
\end{proof}

In Subsection \ref{Subsec4.2} we will prove that if $\Sigma^2$ is a complete, connected, orientable surface of genus $g$ then ${\rm dim} \mathcal H^1_f(\Sigma)\geq 2g$. Hence, if $\Sigma^2$ is $f$-minimal in $\real 3$ and ${\rm Ric}_f\geq \mu>0$  the theorem gives immediately
$$
{\rm Ind}_f(\Sigma)\geq \dfrac{2g}{3}.
$$
which proves the first assertion of Corollary \ref{CorE} in the introduction. If $\Sigma^2$ is a properly immersed shrinker, then 
$$
{\rm Ind}_f(\Sigma)\geq \dfrac{2}{3}g+1.
$$
In fact, this is trivially true if $\Sigma^2$ is a hyperplane (in that case, in fact, the index is equal to $1$); otherwise, we apply Corollary \ref{dimensiontwo}.

\subsection{Proof of Theorem \ref{complete}}

Before giving the proof we state two lemmas.

\begin{lemma}\label{lemma_IBP} Let $\Omega\subset \Sigma$ be a bounded domain and let $\phi\in C^{\infty}_0(\Omega)$. Let $u\in C^{\infty}(\Omega)$. If
$
L_f=\Delta_f +T,
$
is a Schr\"odinger operator and $T\in C^{\infty}(\Omega)$ is any potential, then 
$$
\int_{\Omega}\Big(\abs{\nabla(\phi u)}^2+T\phi^2u^2\Big)e^{-f}\,d\mu =
\int_{\Omega}\phi^2uL_fu\cdot e^{-f}\,d\mu+\int_{\Omega}u^2\abs{\nabla\phi}^2\cdot e^{-f}\,d\mu.
$$
\end{lemma}
\begin{proof}
The proof is obtained using integration by parts and the identity
$$
\Delta_f(uv)=v\Delta_fu+u\Delta_f v-2\scal{\nabla u}{\nabla v}.
$$
\end{proof}

We use a sequence of cut-off functions defined as follows. For each positive integer $n$,  let $B_n$ be the (intrinsic) geodesic ball in $\Sigma$ having radius $n$ and centered at a fixed point $x_0\in\Sigma$. Since $\Sigma$ is complete, it is standard to obtain by \cite[Proposition 2.1]{GW} that there exists a family of smooth functions $\phi_n$ on $\Sigma$ such that $\phi_n=1$ on $B_n$, $\phi_n$ is compactly supported on  $B_{2n}$ and 
$$
\abs{\nabla\phi_n}\leq\dfrac{c}{n}
$$
for a constant $c$ depending only on $\Sigma$. 

\begin{lemma}\label{vusigma} Let $V(\Sigma)=E(\Sigma,\Lambda)+\mathcal{H}^1_f(\Sigma)$. Then
$$
\dim V(\Sigma)=\dim E(\Sigma,\Lambda)+\dim \mathcal{H}^1_f(\Sigma).
$$
Moreover, for any $\eta\in V(\Sigma)$ one has:
$$
\int_{\Sigma}\scal{\Delta^{[1]}_f\eta}{\eta} e^{-f}d\mu\leq\Lambda \int_{\Sigma}\abs{\eta}^2e^{-f}d\mu.
$$
\end{lemma}
\begin{proof} We can  assume that both  spaces are finite dimensional, otherwise the assertion is trivial. The first assertion follows because, if $\xi\in \mathcal{H}^1_f(\Sigma)$ and $u\in W^{1,2}(\Sigma_f)$ then $(\xi,\nabla u)_f=0$, where $(\cdot,\cdot)_f$ is the weighted $L^2$-inner product. Hence $E(\Sigma,\Lambda)$ and $\mathcal{H}^1_f(\Sigma)$ are mutually orthogonal and the assertion follows. 

For the second part, notice that we have
$$
E(\Sigma,\Lambda)=V_{\Delta_f}(\lambda_1)+\dots+V_{\Delta_f}(\lambda_k)
$$
for eigenvalues $0<\lambda_1\leq\dots\leq\lambda_k\leq\Lambda$. One verifies that, if $u\in V_{\Delta_f}(\lambda_i)$ and $v\in V_{\Delta_f}(\lambda_j)$ then:
$$
\int_{\Sigma}\scal{\nabla u}{\nabla v}e^{-f}d\mu=
\lambda_i\int_{\Sigma}uve^{-f}d\mu=\lambda_j\int_{\Sigma}uve^{-f}d\mu.
$$
This allows to construct a (finite) orthonormal basis of $E(\Sigma,\Lambda)$ by eigenfunctions of $\Delta_f$. The standard argument shows that then  any $\eta\in V_{\Sigma}$ satisfies the assertion of the lemma. 
\end{proof}

Let us now prove the theorem. We want to prove that $\dim V(\Sigma)$ is finite and  the following inequality holds:
$$
\mathrm{Ind}_{f}(\Sigma)\geq {{m+1}\choose 2}^{-1}{\rm dim} V(\Sigma).
$$
Set for short:
$$
I=\mathrm{Ind}_{f}(\Sigma), \quad k={\rm dim}V(\Sigma)
$$
with $k$ possibly equal to $+\infty$. 
We have to show that
\begin{equation}\label{prin}
k\leq {{m+1}\choose 2} I.
\end{equation}
Assume by contradiction that $V(\Sigma)$ contains a subspace $E^k$ of dimension $k$ satisying:
$$
k>{{m+1}\choose 2}I.
$$

Consider the exhaustion of $\Sigma$ by relatively compact balls $\{\Omega_n\}\doteq\{B_{2n}\}$ centered at a fixed point $x_0\in\Sigma$. As the index is finite, there exists $n_0$ such that 
$
I=\mathrm{Ind}_{f}(B_{2n})
$
for all $n\geq n_0$. Let $\{\phi_n\}$ be the family of cut-off functions as defined before. 

\medskip

We let $\bar V,\bar W\in\overline{\mathcal{P}}$ and let $V,W$ be their projections on $\Sigma$.  As in the compact case, we introduce the vector field
$$
X_{\bar V,\bar W}=\<\bar V,N\>W-\<\bar W,N\>V.
$$
Define $u=\scal{X_{\bar{V},\bar{W}}}{\xi}$. For each fixed $n\geq n_0$, consider the the family of functions : 
$$
\{u_{n}\}=\{\phi_n u\}.
$$
with $\xi\in E^k$ and $\bar V,\bar W\in\overline{\mathcal{P}}$.
Notice that each such function is zero on the boundary of $\Omega_n\doteq B_{2n}$ and then it can be used as test-functions  for the stability operator of $\Omega_n$. We now proceed exactly as in the compact case. Consider the first $I$ eigenfunctions of the stability operator on $\Omega_n$, say $\{f_1,f_2,\dots,f_I\}$. 
We look for non-zero vector fields $\xi\in E^k$ such that the following orthogonality relations hold for all possible choices of $\bar V,\bar W$:
$$
\int_{\Sigma}\phi_n\scal{X_{\bar V,\bar W}}{\xi}f_1e^{-f}d\mu=\dots=\int_{\Sigma}\phi_n\scal{X_{\bar V,\bar W}}{\xi}f_Ie^{-f}d\mu=0.
$$
This is a system of ${{m+1}\choose 2}I$ homogeneous linear equations in the unknown $\xi\in E^k$. Counting dimensions we see that, as
$
k>{{m+1}\choose 2}I
$
by our assumption, we can find a non-trivial vector field $\xi_n\in E^k$ which verifies all of those equations for all choices of $\bar V,\bar W$. 

The notation $\xi_n$ stresses the fact that such vector field depends on $n$. We can choose $\xi_n$ so that it has unit $L^{2}(\Sigma_f)$-norm:
$$
\int_{\Sigma}\abs{\xi_n}^2e^{-f}\,d\mu=1.
$$
In what follows, we make the identification: 
$$
\{\xi\in E^k: \int_{\Sigma}\abs{\xi}^2 e^{-f}\,d\mu=1\}=\sphere{k-1},
$$
 in particular, we can think of $\xi_n$ as an element of $\sphere{k-1}$.

\medskip

As $h$ is the index of $\Omega_n$, we see that $\lambda_{I+1}\geq 0$ and hence
$$
\int_{\Omega_n}\Big(\abs{\nabla(\phi_nu)}^2-(\abs{A}^2+\mathrm{Hess}f(N,N))\phi_n^2u^2\Big)\cdot e^{-f}d\mu\geq \lambda_{I+1}\int_{\Omega_n}\phi_n^2u^2\,e^{-f}d\mu\geq 0
$$
for all $u=\scal{X_{\bar V,\bar W}}{\xi_n}$.

We apply Lemma \ref{lemma_IBP} to the above inequality with $\Omega=\Omega_n$, $\phi=\phi_n$ and  $T=-(\abs{A}^2+\mathrm{Hess}f(N,N))$. We thus obtain
$$
0\leq \int_{\Omega_n}\phi_n^2uL_fu\cdot e^{-f}\,d\mu+\int_{\Omega_n}u^2\abs{\nabla\phi_n}^2\cdot e^{-f}\,d\mu
$$
which again, is valid for all $\bar V,\bar W$. Proceeding as in the compact case, integrating with respect to $\bar V,\bar W$  we obtain:
$$
\begin{aligned}
0\leq &-\int_{\Omega_n}\phi_n^2\mathrm{Hess}f(\xi_n,\xi_n)e^{-f}\,d\mu-\int_{\Omega_n}\phi_n^2
\mathrm{Hess}f(N,N)\abs{\xi_n}^2e^{-f}\,d\mu\\
&+ \int_{\Omega_n}\phi_n^2\scal{\Delta^{[1]}_f\xi_n}{\xi_n}e^{-f}d\mu+ \int_{\Omega_n}\abs{\nabla\phi_n}^2\abs{\xi_n}^2\cdot e^{-f}\,d\mu.
\end{aligned}
$$
Recalling the properties of the cut-off functions $\phi_n$ (namely $\phi_n=1$ on $B_n$), using the hypothesis $\mathrm{Hess}f\geq\mu>0$ and the inequality in Lemma \ref{vusigma}, we obtain, for all $n\geq n_0$ :
\begin{equation}
\begin{aligned}
(2\mu-\Lambda) \int_{B_n}\abs{\xi_n}^2\cdot e^{-f}\,d\mu&\leq \dfrac{c^2}{n^2}\int_{B_{2n}}\abs{\xi_n}^2\cdot e^{-f}\,d\mu\\
&\leq \dfrac{c^2}{n^2}.
\end{aligned}
\end{equation}
As by assumption $2\mu-\Lambda>0$ we see
\begin{equation}\label{absurd}
\lim_{n\to\infty}\int_{B_n}\abs{\xi_n}^2\cdot e^{-f}\,d\mu=0.
\end{equation}
However, we will show below that this can't hold, thus getting a contradiction. The contradiction comes from the assumption 
$
k>{{m+1}\choose 2}I.
$
Hence
$$
k\leq {{m+1}\choose 2}I,
$$
as asserted. 

\medskip

Let us then show that \eqref{absurd} can't hold. By the compactness of $\sphere{k-1}$, the infinite set $\{\xi_n\in\sphere{k-1}: n\geq n_0\}$ has an accumulation point, hence there exists a subsequence $\{\xi_{n_j}\}_{j=1,2,\dots}\in\sphere{k-1}$ which converges to $\xi\in\sphere{k-1}\subseteq E^k$ in the $L^{2}(\Sigma_f)$-sense as $j\to\infty$:
$$
\lim_{j\to\infty}\xi_{n_j}=\xi.
$$
We will presently show that
\begin{equation}\label{limitone}
\lim_{j\to\infty}\int_{B_{n_j}}\abs{\xi_{n_j}}^2e^{-f}\,d\mu=1,
\end{equation}
which will contradict \eqref{absurd}.
Introduce the notation
$$
\norm\xi_{\Sigma}^2=\int_{\Sigma}\abs{\xi}^2e^{-f}\,d\mu,\quad \norm\xi_{B_{n_j}}^2=\int_{B_{n_j}}\abs{\xi}^2e^{-f}\,d\mu
$$
Clearly $\norm\xi_{\Sigma}\geq \norm\xi_{B_{n_j}}$ and 
$
\lim_{j\to\infty}\norm{\xi_{n_j}-\xi}_{\Sigma}=0,
$
by assumption. 
Now, for all $n\geq n_0$:
$$
\int_{B_{n_j}}\abs{\xi_{n_j}}^2e^{-f}\,d\mu
=\int_{B_{n_j}}(\abs{\xi_{n_j}}^2-\abs{\xi}^2)e^{-f}\,d\mu+\int_{B_{n_j}}\abs{\xi}^2e^{-f}\,d\mu
$$
As $\lim_{n\to\infty}\int_{B_{n_j}}\abs{\xi}^2e^{-f}\,d\mu=1$, for \eqref{limitone} to be true  it is enough to show:
\begin{equation}\label{limittwo}
\lim_{j\to\infty}\left |\int_{B_{n_j}}(\abs{\xi_{n_j}}^2-\abs{\xi}^2)e^{-f}\,d\mu\right|=0.
\end{equation}

For any pair $v,w$ of unit vectors in an inner product space one has the inequality
$$
\left|\norm{v}^2-\norm{w}^2\right|\leq 2\norm{v-w}.
$$
Therefore
$$
\lim_{j\to\infty}\left|\norm{\xi_{n_j}}_{B_{n_j}}^2-\norm{\xi}_{B_{n_j}}^2\right|\leq 
2\lim_{j\to\infty}\norm{\xi_{n_j}-\xi}_{B_{n_j}}\leq
2\lim_{j\to\infty}\norm{\xi_{n_j}-\xi}_{\Sigma}=0
$$
and \eqref{limittwo} follows.



\subsection{Proof of Theorem \ref{twog}}\label{Subsec4.2}
Recall that  $\Sigma$ is a complete, connected, orientable surface having genus $g$. If  $\mathcal H^1_f(\Sigma)$ denotes the space of $L^{2}_{f}$ $f$-harmonic one-forms on $\Sigma$, then we have to show that
$$
\dim\mathcal H^1_f(\Sigma)\geq 2g.
$$

{\bf Step 1.} To any closed curve $\gamma$ on $\Sigma$, we  associate a closed $1$-form $\eta_{\gamma}$, whose support is close to $\gamma$.

\smallskip

In fact, as $\Sigma$ is orientable, we have a global unit normal vector field $\nu$ on $\gamma$. Let $U_{2\epsilon}$ be the set of points at distance less than $2\epsilon$ to $\gamma$, on the side defined by $\nu$.  If $\epsilon$ is small enough, this set is an annulus, and we can construct a smooth function 
$f:\Sigma\setminus\gamma\to \reals$ such that
$$
f=\twosystem
{1\quad\text{on}\quad U_{\epsilon}\setminus\gamma}
{0\quad\text{on}\quad (\Sigma\setminus\gamma)\setminus U_{2\epsilon}}
$$
Now set:
$$
\eta_{\gamma}=\twosystem
{df\quad\text{on}\quad \Sigma\setminus\gamma}
{0\quad\text{on}\quad \gamma}
$$
Then, $\eta_{\gamma}$ is a smooth, globally defined 1-form, supported on the compact set given by the closure of $U_{2\epsilon}\setminus U_{\epsilon}$. Clearly, $\eta_{\gamma}$ is closed. 

\medskip

{\bf Step 2.} Assume that $\Sigma\setminus\gamma$ is connected; then, we can 
easily construct a closed curve $\tilde\gamma$ meeting $\gamma$ only at one point $p$, and intersecting it transversally. We fix such a  curve $\tilde\gamma$ and  call it the {\it dual curve} of $\gamma$. It is also easy to verify that, if $\eta_{\gamma}$ is the one-form associated to $\gamma$, then, for a suitable orientation of $\gamma$:
$$
\int_{\tilde\gamma}\eta_{\gamma}=1.
$$
In fact, pick points $p_+,p_-\in\tilde\gamma$ very close to $p=\tilde\gamma\cap\gamma$ and laying on the two opposite sides of $\gamma$. The integral of $\eta_{\gamma}$ on the arc joining $p_+$ and $p_-$, not intersecting $\gamma$, is $f(p_+)-f(p_-)$, which is $1$ when the two points are sufficiently close to $p$. Taking the limit shows the assertion. 

\medskip

{\bf Step 3.} Since the genus of $\Sigma$ is $g$, we can find $g$ disjoint closed curves $\gamma_1,\dots,\gamma_g$ such that the set  $\Sigma\setminus(\gamma_1\cup\dots\cup\gamma_g)$ is connected. Construct the associated $1$-forms  $\eta_1,\dots,\eta_g$ so that they have mutually disjoint support (this is certainly possible, as explained before).
To each $\gamma_j$ we associate its dual closed curve $\tilde\gamma_j$, as above; looking at the process, we can do it so that this family of curves has the following properties:

\nero $\tilde\gamma_j$ intersects $\gamma_j$ only once, transversally, and 

\nero $\tilde\gamma_j$ does not intersect any of the other curves $\gamma_k$, with $k\ne j$.
\medskip

The first property gives $\int_{\tilde\gamma_j}\eta_j=1$; the second property (after eventually restricting the support of each $\eta_j$) gives $\int_{\tilde\gamma_j}\eta_k=0$ for $j\ne k$. In conclusion we have:
\begin{equation}\label{path}
\int_{\tilde\gamma_j}\eta_k=\delta_{jk} \quad\text{for all $j,k=1,\dots,g$}.
\end{equation}

{\bf Step 4.} We now use the Hodge orthogonal decomposition, due to E. L. Bueler \cite[Theorem 5.7]{Bue}.
$$
L^{2}_f(\Lambda^1(\Sigma))=A\oplus B_f\oplus H^1_f(\Sigma),
$$
where
$$
\threesystem
{A=\overline{\{dg: \text{$g\in C^{\infty}(\Sigma)$ has compact support}\}}}
{B_f=\overline{\{\delta_f\psi: \text{$\psi\in\Lambda^2(\Sigma)$ has compact support}\}}}
{\Cal{H}^1_f(\Sigma)=\{\omega\in L^{2}_f(\Lambda^1(\Sigma)): d\omega=\delta_f\omega=0\}}
$$
the closure being taken in the $L^2_f$-norm. We denote by $P$ the orthogonal projection onto $\Cal{H}^1_f(\Sigma)$.

\medskip

{\bf Step 5.} We now consider the $L^2_f$ $f$-harmonic forms:
$$
\omega_j\doteq P(\eta_j); \quad j=1,\dots,g
$$ 
and claim that they are linearly independent. In fact, as $\eta_j$ is closed, we have 
$\omega_j=\eta_j+dg_j$ with $dg_j\in A$. Hence, by \eqref{path}, 
$$
\int_{\tilde\gamma_j}\omega_k=\delta_{jk}.
$$
If $\sum_{j=1}^ga_j\omega_j=0$, integrating both sides on the dual curves gives $a_1=\dots=a_g=0$.

\medskip

{\bf Step 6.} At this point, we need $g$ more linearly independent $L^2_f$-harmonic $1$-forms. In the unweighted case, one simply takes the Hodge-star dual of the forms in the previous step. Recall that, if $(e_1,e_2)$ be a positively oriented local orthonormal frame with dual frame $(\theta_1,\theta_2)$, then  the classical Hodge $\star$-operator is defined by the rule:
$$
\omega\wedge\star\eta=\scal{\omega}{\eta}\theta_1\wedge\theta_2.
$$
Then one checks that $\star d\omega=\delta\star\omega$, and $d\star\omega=-\star\delta\omega$. If $f=0$ and $\omega$ is harmonic, then $\star\omega$ is also harmonic. However, this is no longer true in the weighted case and we introduce the weighted Hodge star operator $\star_f$ as follows.  For any one-form $\omega$:
$$
\star_f\omega\doteq e^f\star\omega.
$$
The following lemma has a straightforward proof, which we omit.

\begin{lemma}\label{hodgestar} One has
$$
\delta_f\star_f\omega=\star_fd\omega, \quad d\star_{-f}\omega=-\star_{-f}\delta_f\omega.
$$
\end{lemma}
\medskip

{\bf Step 7.} Define, for $j=1,\dots,g$, 
$$
\zeta_j=P(\star_f\eta_j).
$$
Note that this is well-defined because $\star_f\eta_j$, being compactly supported, is in $L^2_f$. We need to show that the $f$-harmonic forms $\zeta_1,\dots,\zeta_g$ are linearly independent. To that end, observe that, by Lemma \ref{hodgestar}, 
$\star_f\eta_j$ is $f$-coclosed, hence $\zeta_j=\star_f\eta_j+\delta_f\psi_j$ for some $\delta_f\psi_j\in B_f$.  Assume $\sum_ja_j\zeta_j=0$; then $\sum_ja_j(\star_f\eta_j)=\delta_f\psi$ for some $\delta_f\psi\in B_f$. Applying $\star_{-f}$ on both sides, using Lemma \ref{hodgestar} and noting that $\star_{-f}\star_f=\star^2=-1$
we see:
$$
\sum_ja_j\eta_j=-\star_{-f}\delta_f\psi=d\star_{-f}\psi
$$
that is, the form on the left is exact. Hence, integrating both sides on the dual curves $\tilde\gamma_1,\dots,\tilde\gamma_g$ we get $a_1=\dots=a_g$. The conclusion is that $\zeta_1,\dots,\zeta_g$ are linearly independent.

\medskip

{\bf Step 8.} Now consider the subspace $E\subseteq H^1_f(\Sigma)$ given by
$$
E={\rm span}(\omega_1,\dots,\omega_g,\zeta_1,\dots,\zeta_g).
$$
We prove that ${\rm dim}E=2g$, which will imply the final assertion. It is clear that
it is enough to show that 
$$
(\omega_j,\zeta_k)_f=0
$$
for all $j,k$. This is done  as in the unweighted case.

\medskip

Recall that $\omega_j=\eta_j+dg_j$, and $\zeta_k=\star_f\eta_k+\delta_f\phi_k$, for $dg_j\in A$ and $\delta_f\phi_k\in B_f$.  Then:
$$
(\omega_j,\zeta_k)_f=(\eta_j,\star_f\eta_k)_f+(\eta_j,\delta_f\phi_k)_f+(dg_j,\star_f\eta_k)_f+
(dg_j,\delta_f\phi_k)_f.
$$
Now $(dg_j,\delta_f\phi_k)_f=0$ because the Hodge-Bueler decomposition is orthogonal. Next, a compactly supported closed form is always orthogonal to an f-coexact form (use Green formula)  hence $(\eta_j,\delta\phi_k)_f=0$. Similary, as $\star_f\eta_k$ is f-coclosed and compactly supported, we see
that $(dg_j,\star_f\eta_k)_f=0$. We end-up with
$$
(\omega_j,\zeta_k)_f=(\eta_j,\star_f\eta_k)_f.
$$
If $j=k$ this is zero because $\eta_j$ and $\star_f\eta_j$ are pointwise orthogonal. If $j\ne k$ this is zero because $\eta_j$ and $\star_f\eta_k$ have disjoint support. The assertion follows.

\end{document}